\documentclass[12pt]{article}
\usepackage[left=2.5cm,right=2.5cm,top=2.5cm,bottom=2.5cm,a4paper]{geometry}

\usepackage[a4paper]{geometry}
\usepackage{graphicx}
\usepackage{microtype}
\usepackage{siunitx}
\usepackage{booktabs}
\usepackage{cleveref}
\usepackage{graphics}
\usepackage{graphicx}
\usepackage{epsfig}
\usepackage{amsmath,amsfonts,amssymb,amsthm}
\usepackage{listings}
\usepackage{paralist}
\usepackage{sectsty}
\numberwithin{equation}{section}
\date{}
\subsectionfont{\normalfont}
\makeatother

\newcommand{\R}{{\mathbb R}}
\newcommand{\Q}{{\mathbb Q}}

\newcommand{\Z}{{\mathbb Z}}

\theoremstyle{plain}
\numberwithin{equation}{section}
\newtheorem{thm}{Theorem}[section]
\newtheorem{theorem}[thm]{Theorem}
\newtheorem{lemma}[thm]{Lemma}
\newtheorem{corollary}[thm]{Corollary}
\newtheorem{remark}[thm]{Remark}
\newtheorem{example}[thm]{Example}
\newtheorem{question}[thm]{Question}
\newtheorem{definition}[thm]{Definition}

\newtheorem{thm_A}{Theorem}
\newtheorem{theorem_A}[thm_A]{Theorem}
\newtheorem{lemma_A}[thm_A]{Lemma}
\newtheorem{corollary_A}[thm_A]{Corollary}
\newtheorem{remark_A}[thm_A]{Remark}

\begin{document}
%% replace the values in the next three lines by the correct information
\setcounter{page}{1}

\title{On a generalized $k$-FL sequence and its applications}
\author{WonTae Hwang$^{1}$, Youngwoo Kwon$^{2}$ and Kyunghwan Song$^{2,*}$\\
	${}^{1}$School of Mathematics, Korea Institute for Advanced Study, \\
	Seoul 02455, Republic of Korea\\
	${}^{2}$Department of Mathematics, Korea University, \\
	Seoul 02841, Republic of Korea\\
	${}^{*}${Corresponding Author Email: heroesof@korea.ac.kr}
}

\maketitle

\begin{abstract}
We introduce a generalized $k$-FL sequence and special kind of pairs of real numbers that are related to it. As an application, we describe a family of integer sequences each of whose element has a finite Zsigmondy set. Also, we associate skew circulant and circulant matrices to each generalized $k$-FL sequence, and study the determinantal variety of those matrices as another application.
\end{abstract}

    \section{Introduction} \label{sec_intro}
Given an integer sequence $\mathcal{A}=\{a_n \}_{n \in \Z_{\geq 0}},$ we may introduce the Zsigmondy set, denoted $\mathcal{Z}(\mathcal{A})$, associated to the sequence. It might be an interesting question to ask whether the Zsigmondy set of a given integer sequence is finite or not. In the literature, there are many examples of an integer sequence $\mathcal{A}$ such that $\mathcal{Z}(\mathcal{A})$ is finite. For a brief survey on some of such examples, see Application 1. Now, we can generalize this concept to a family of integer sequences in the following sense: let $\mathcal{F}$ be a family of integer sequences. We let
\begin{equation*}
N_{\mathcal{F}}=\sup_{\mathcal{A}:=\{a_n\}_{n \in \Z_{\geq 0}} \in \mathcal{F}} |\mathcal{Z}(\mathcal{A})|
\end{equation*}
where $|S|$ denotes the cardinality of a set $S.$ We can ask whether $N_{\mathcal{F}}$ is finite or not. More precisely, we can ask the following question:
\begin{question}\label{que_Zsigmondy_1}
	Is there any family $\mathcal{F}$ of integer sequences such that $N_{\mathcal{F}}$ is infinite but each element of $\mathcal{F}$ has a finite Zsigmondy set?
\end{question}

In this paper, in an attempt to give an answer for this question, we will introduce a new sequence, called a \emph{generalized $k$-FL sequence}, that is defined by an $\R$-linear combination of the $k$-Fibonacci sequence and the $k$-Lucas sequence for a real number $k$. We recall that the $k$-Fibonacci sequence is a generalization of the classical Fibonacci sequence, and is defined recursively by
\begin{equation*}
F_{k,0}=0, F_{k,1}=1,~\textrm{and}~F_{k,n}=k \cdot F_{k,n-1} + F_{k,n-2}~~~\textrm{for}~n \geq 2
\end{equation*}
where $k$ is a positive real number. This $k$-Fibonacci sequence was introduced and used for the well-known four-triangle longest-edge (4TLE) partition\cite{SA}. Similarly, the $k$-Lucas sequence is a generalization of the classical Lucas sequence, and is defined recursively by
\begin{equation*}
L_{k,0}=2, L_{k,1}=k,~\textrm{and}~L_{k,n}=k \cdot L_{k,n-1} + L_{k,n-2}~~~\textrm{for}~n \geq 2
\end{equation*}
where $k$ is a positive real number, that was introduced in \cite{S}. \\

After defining a generalized $k$-FL sequence in terms of the $k$-Fibonacci sequence and the $k$-Lucas sequence, we will associate skew circulant and circulant matrices to each generalized $k$-FL sequence, and then examine the invertibility of such matrices to define special kind of pairs of real numbers. To define such pairs, we also need the well-known fact that the plane curve defined by the equation $y^2=x^3+ax+b$ ($a,b \in \R$) is singular if and only if $4a^3+27b^2=0.$ Finally, as applications, we will give an answer for Question \ref{que_Zsigmondy_1}, and describe the determinantal variety of those associated matrices. \\

This paper is organized as follows: \\

In Section \ref{sec_matrix}, we define a $k$-FL sequence as a linear combination of the $k$-Fibonacci sequence and the $k$-Lucas sequence for a positive real number $k$, which also depends on two positive real numbers $a,b$. Also, we introduce a generalized $k$-FL sequence by weakening the positivity condition on the real numbers $k,a,$ and $b$ above, together with the associated skew circulant and circulant matrices to a generalized $k$-FL sequence. In particular, we give the criteria for those matrices to be invertible. \\

In Section \ref{sec_pair}, we introduce special kind of pairs of real numbers, called \emph{singular $k$-FL pairs} of various type, that are related to the non-invertibility of the associated skew circulant and circulant matrices to a given generalized $k$-FL sequence. Also, we describe a set of real numbers that encode certain property of a singular plane curve. \\

In Section \ref{sec_application},  we obtain some information about a family of integer sequences each of whose element has a finite Zsigmondy set, using generalized $k$-FL sequences, to answer Question \ref{que_Zsigmondy_1}. Also, given a generalized $k$-FL sequence, we use singular $k$-FL pairs of various type to give a description of the determinantal variety of the associated skew circulant and circulant matrices, as applications.  \\

Finally, we provide some identities related to a $k$-FL sequence in the Appendix.
%Acknowledgments may be included at the end of the Introduction.

\section{Generalized $k$-FL Sequences and Associated Matrices}\label{sec_matrix}
We start this section by considering a new sequence that is obtained by taking a linear combination of a $k$-Fibonacci sequence and a $k$-Lucas sequence for some positive real number $k$.
\begin{definition}
	Let $k,a$, and $b$ be positive real numbers. We define a \emph{$k$-FL sequence} to be a sequence $\{S_{k,n}^{(a,b)}\}_{n \in \Z_{\geq 0}}$ given by the recurrence relation
	\begin{equation}
	S_{k,0}^{(a,b)}=2b, S_{k,1}^{(a,b)}=bk+a,~\textrm{and}~S_{k,n}^{(a,b)}=k \cdot S_{k,n-1}^{(a,b)}+S_{k,n-2}^{(a,b)}~\textrm{for}~n \geq 2.
	\end{equation}
	In other words, $S^{(a,b)}_{k,n}=a \cdot  F_{k,n}+b \cdot L_{k,n}$ for each integer $n \geq 0$, where $F_{k,n}$ (resp. $L_{k,n}$) denotes the $n$-th $k$-Fibonacci number (resp. the $n$-th $k$-Lucas number).
\end{definition}

We assumed the positivity of the real numbers $k,a$, and $b$ in the above definition only because we wanted the sequence to be increasing as in the case of the classical Fibonacci sequence. By removing the positivity condition on those numbers, we have the following
\begin{definition}[Generalized $k$-FL sequence]
	Let $k,a$, and $b$ be real numbers. We define a \emph{generalized $k$-FL sequence} to be a sequence $\{S_{k,n}^{(a,b)}\}_{n \in \Z_{\geq 0}}$ given by the recurrence relation
	\begin{equation}\label{eq_k-FL}
	S_{k,0}^{(a,b)}=2b, S_{k,1}^{(a,b)}=bk+a,~\textrm{and}~S_{k,n}^{(a,b)}=k \cdot S_{k,n-1}^{(a,b)}+S_{k,n-2}^{(a,b)}~\textrm{for}~n \geq 2.
	\end{equation}
\end{definition}
\begin{example}
	For any real numbers $a$ and $b$, the sequence $\{S_{0,n}^{(a,b)}\}_{n \in \Z_{\geq 0}}$ has the property that
	\begin{equation*}
	S_{0,2i}^{(a,b)}=2b~~\textrm{and}~~S_{0,2i+1}^{(a,b)}=a~~\textrm{for all}~i \geq 0.
	\end{equation*}
\end{example}

Now, given a generalized $k$-FL sequence $\{S_{k,n}^{(a,b)}\}_{n \in \Z_{\geq 0}}$, we may introduce two associated matrices:
\begin{equation*}
A^{(k,a,b)}_n :=\left(\begin{array}{cccc} S^{(a,b)}_{k,1} & S^{(a,b)}_{k,2} & \cdots & S^{(a,b)}_{k,n} \\ -S^{(a,b)}_{k,n} & S^{(a,b)}_{k,1} &\cdots & S^{(a,b)}_{k,n-1} \\ \vdots & \vdots & \vdots & \vdots \\ -S^{(a,b)}_{k,2} & -S^{(a,b)}_{k,3} & \cdots & S^{(a,b)}_{k,1}
\end{array}\right)
\end{equation*}
and
\begin{equation*}
B^{(k,a,b)}_n :=\left(\begin{array}{cccc} S^{(a,b)}_{k,1} & S^{(a,b)}_{k,2} & \cdots & S^{(a,b)}_{k,n} \\ S^{(a,b)}_{k,n} & S^{(a,b)}_{k,1} &\cdots & S^{(a,b)}_{k,n-1} \\ \vdots & \vdots & \vdots & \vdots \\ S^{(a,b)}_{k,2} & S^{(a,b)}_{k,3} & \cdots & S^{(a,b)}_{k,1}
\end{array}\right)
\end{equation*}
for each integer $n \geq 1$. (Note that $A_n^{(k,a,b)}$ is a skew circulant matrix and $B_{n}^{(k,a,b)}$ is a circulant matrix.)
\begin{example}
	If $a=b=0$, then both $A^{(k,a,b)}_{n}$ and $B^{(k,a,b)}_{n}$ are the zero matrix for any real number $k$ and any integer $n \geq 1$.
\end{example}
Regarding the invertibility of these matrices, we have the following crucial
\begin{theorem}\label{thm_matrix_invertible_1}
	Let $k \ne 0$ be a real number.
	\begin{enumerate}
		\item[(a)] For an odd integer $n \geq 1,$ the matrix $A^{(k,a,b)}_{n}$ is invertible if and only if $S_{k,n+1}^{(a,b)} - S_{k,n}^{(a,b)} \ne -bk+2b-a$.
		\item[(b)] For an even integer $n \geq 2,$ the matrix $B^{(k,a,b)}_{n}$ is invertible if and only if $S_{k,n+1}^{(a,b)} - S_{k,n}^{(a,b)} \ne bk-2b+a$ and $S_{k,n+1}^{(a,b)}+S_{k,n}^{(a,b)} \ne bk+2b+a$.
	\end{enumerate}
\end{theorem}
To prove the theorem, we need
\begin{lemma}\label{lem_matrix_invertible_skewcirculant}
	Let $n \geq 1$ be an integer.
	\begin{enumerate}
		\item[(a)] Let $A=\left(\begin{array}{cccc} a_1 & a_2 & \cdots & a_n \\ -a_n & a_1 &\cdots & a_{n-1} \\ \vdots & \vdots & \vdots & \vdots \\ -a_{2} &  -a_{3} & \cdots & a_{1}
		\end{array}\right)\in M_{n\times n}(\R)$ be a skew circulant matrix. Then $A$ is invertible if and only if $f(\omega^{m} \eta) \ne 0$ for all $m=0,1,\cdots,n-1$, where $\displaystyle f(x)=\sum_{j=1}^{n} a_{j} x^{j-1} \in \R[x]$, $\displaystyle \omega = \textrm{exp} \left(\frac{2\pi \sqrt{-1}}{n}\right)$, and $\eta=\textrm{exp} \left(\frac{\pi \sqrt{-1}}{n}\right)$.
		\item[(b)] Let $B=\left(\begin{array}{cccc} a_{1} & a_{2} & \cdots & a_{n} \\ a_{n} & a_{1} &\cdots & a_{n-1} \\ \vdots & \vdots & \vdots & \vdots \\ a_{2} &  a_{3} & \cdots & a_{1}
		\end{array}\right)\in M_{n\times n}(\R)$ be a circulant matrix. Then $B$ is invertible if and only if $f(\omega^{m}) \ne 0$ for all $m=0,1,\cdots,n-1$, where $\displaystyle f(x)=\sum_{j=1}^{n} a_{j} x^{j-1} \in \R[x]$ and $\displaystyle \omega = \textrm{exp} \left(\frac{2\pi \sqrt{-1}}{n}\right)$.
	\end{enumerate}
\end{lemma}
\begin{proof}
	\begin{enumerate}
		\item[(a)] For a proof, see \cite[Lemma 3]{GJG}.
		\item[(b)] For a proof, see \cite[Lemma 1.1]{SCH}.
	\end{enumerate}
\end{proof}
Now, we are ready to prove the theorem.
\begin{proof}[Proof of Theorem \ref{thm_matrix_invertible_1}]
	As in the proof of Theorem \ref{thm_catalan} (see Appendix below), we can write
	\begin{equation*}
	S_{k,n}^{(a,b)}=\frac{X \alpha^{n} - Y \beta^{n}}{\alpha - \beta}
	\end{equation*}
	where $X=a+\sqrt{k^{2}+4}b$ and $Y=a-\sqrt{k^{2}+4}b$. Also, let $f(x) \in \R[x]$ be a polynomial defined as in Lemma \ref{lem_matrix_invertible_skewcirculant}.
	\begin{enumerate}
		\item[(a)] We have
		\begin{align}
		f(\omega^{m} \eta) & =\sum_{j=1}^{n} S_{k,j}^{(a,b)}(\omega^{m} \eta)^{j-1} \nonumber\\
		& =\sum_{j=1}^{n} \frac{X \alpha^{j}- Y \beta^{j}}{\alpha- \beta}(\omega^{m} \eta)^{j-1} \nonumber\\
		& =\frac{X \alpha (1+\alpha^{n} )(1-\beta \omega^{m} \eta)-Y \beta (1+\beta^{n})(1-\alpha \omega^{m} \eta)}{(\alpha-\beta) (1-\alpha \omega^{m} \eta) (1-\beta \omega^{m} \eta)} \nonumber\\
		& =\frac{(X \alpha - Y \beta)+X(\alpha^{n+1}-Y \beta^{n+1})+\omega^{m} \eta (X-Y) + \omega^{m} \eta (X \alpha^{n} - Y \beta^{n} )}{(\alpha-\beta) (1-\alpha \omega^{m} ) (1-\beta \omega^{m})} \nonumber\\
		& =\frac{bk+a+S_{k,n+1}^{(a,b)}+\omega^{m} \eta (2b+S_{k,n}^{(a,b)})}{1- k \omega^{m} \eta - \omega^{2m}\eta^{2} }\label{eq_invertible_skewcirculant}
		\end{align}
		for all $m=0,1,\cdots,n-1$. \\
		Now, suppose on the contrary that $S_{k,n+1}^{(a,b)} - S_{k,n}^{(a,b)} = -bk+2b-a$. Then by taking $m=\frac{n-1}{2}$, we get from (\ref{eq_invertible_skewcirculant}) that
		\begin{equation*}
		f(\omega^{m} \eta)=f\left(\omega^{\frac{n-1}{2}} \eta \right)=f(-1)=\frac{bk+a+S_{k,n+1}^{(a,b)}-2b -S_{k,n}^{(a,b)}}{k} = 0
		\end{equation*}
		and by Lemma \ref{lem_matrix_invertible_skewcirculant}-(a), this contradicts the assumption that $A_{n}^{(k,a,b)}$ is invertible. \\
		Conversely, suppose that there exists an $l \in \{0,1,\cdots,n-1\}$ such that $f(\omega^{l} \eta)=0$. From (\ref{eq_invertible_skewcirculant}), it follows that
		\begin{equation*}
		bk+a + S_{k,n+1}^{(a,b)} + \omega^{l} \eta (2b + S_{k,n}^{(a,b)})=0.
		\end{equation*}
		Note that, by our assumption, we have
		\begin{equation*}
		S_{k,n}^{(a,b)}+2b \ne 0
		\end{equation*}
		and hence
		\begin{equation}\label{eq_omega_l_eta}
		\omega^l \eta =-\frac{bk+a+S_{k,n+1}^{(a,b)}}{S_{k,n}^{(a,b)}+2b}.
		\end{equation}
		Then since $\omega^{l} \eta = \textrm{exp} \left(\frac{(2l+1)\pi \sqrt{-1}}{n}\right)$ and the (RHS) of (\ref{eq_omega_l_eta}) is a real number, the only possibility is that $\omega^{l} \eta = -1.$ Then it follows from (\ref{eq_omega_l_eta}) that
		\begin{equation*}
		S_{k,n+1}^{(a,b)}-S_{k,n}^{(a,b)}=-bk +2b - a,
		\end{equation*}
		which contradicts our assumption.
		\item[(b)] Similarly, we have
		\begin{align}
		f(\omega^{m}) & =\sum_{j=1}^{n} S_{k,j}^{(a,b)}(\omega^{m} )^{j-1} \nonumber\\
		& =\sum_{j=1}^{n} \frac{X \alpha^{j}- Y \beta^{j}}{\alpha- \beta}(\omega^{m})^{j-1} \nonumber \\
		& =\frac{X \alpha (1-\alpha^{n} )(1-\beta \omega^{m} )-Y \beta (1-\beta^{n})(1-\alpha \omega^{m} )}{(\alpha-\beta) (1-\alpha \omega^{m} ) (1-\beta \omega^{m})} \nonumber \\
		& =\frac{(X \alpha - Y \beta)-X(\alpha^{n+1}-Y \beta^{n+1})+\omega^{m} (X-Y) - \omega^{m} (X \alpha^{n} - Y \beta^{n} )}{(\alpha-\beta) (1-\alpha \omega^{m} ) (1-\beta \omega^{m})} \nonumber \\
		& =\frac{bk+a-S_{k,n+1}^{(a,b)}+\omega^m (2b-S_{k,n}^{(a,b)})}{1- k \omega^{m} - \omega^{2m}}\label{eq_invertible_circulant}
		\end{align}
		for all $m=0,1,\cdots,n-1$. \\
		Now, suppose on the contrary that either $S_{k,n+1}^{(a,b)} - S_{k,n}^{(a,b)} = bk-2b+a$ or $S_{k,n+1}^{(a,b)}+S_{k,n}^{(a,b)} = bk+2b+a$.
		\begin{enumerate}
			\item[(i)] If $S_{k,n+1}^{(a,b)} - S_{k,n}^{(a,b)} = bk-2b+a$, then by taking $m=\frac{n}{2}$, we get from (\ref{eq_invertible_circulant}) that
			\begin{equation*}
			f(\omega^{m})=f\left(\omega^{\frac{n}{2}}\right)=f(-1)=\frac{bk+a-S_{k,n+1}^{(a,b)}-2b +S_{k,n}^{(a,b)}}{k} = 0
			\end{equation*}
			and by Lemma \ref{lem_matrix_invertible_skewcirculant}-(b), this contradicts the assumption that $B_{n}^{(k,a,b)}$ is invertible. \\
			\item[(ii)] If $S_{k,n+1}^{(a,b)} + S_{k,n}^{(a,b)} = bk+2b+a$, then by taking $m=0$, we get from (\ref{eq_invertible_circulant}) that
			\begin{equation*}
			f(\omega^{m})=f(1)=\frac{bk+a-S_{k,n+1}^{(a,b)}+2b -S_{k,n}^{(a,b)}}{-k} = 0
			\end{equation*}
			and again, by Lemma \ref{lem_matrix_invertible_skewcirculant}-(b), this contradicts the assumption that $B_{n}^{(k,a,b)}$ is invertible.
		\end{enumerate}
		Conversely, suppose that there exists an $l \in \{0,1,\cdots,n-1\}$ such that $f(\omega^{l})=0$. From (\ref{eq_invertible_circulant}), it follows that
		\begin{equation*}
		bk+a - S_{k,n+1}^{(a,b)} + \omega^l (2b - S_{k,n}^{(a,b)})=0.
		\end{equation*}
		Note that, by our assumption, we have
		\begin{equation*}
		S_{k,n}^{(a,b)}-2b \ne 0
		\end{equation*}
		and hence
		\begin{equation}\label{eq_omega_l}
		\omega^{l} =\frac{bk+a-S_{k,n+1}^{(a,b)}}{S_{k,n}^{(a,b)}-2b}.
		\end{equation}
		Then since $\omega^{l} = \textrm{exp} \left(\frac{2 \pi l \sqrt{-1}}{n}\right)$ and the (RHS) of (\ref{eq_omega_l}) is a real number, we can see that we have either $\omega^{l} = 1$ or $\omega^{l} = -1$.
		\begin{enumerate}
			\item[(i)] If $\omega^{l} = 1$, then, by Theorem \ref{thm_identity_sum_all} (see Appendix below), we have
			\begin{equation*}
			f(1)=\sum_{j=1}^{n} S_{k,j}^{(a,b)}=\frac{1}{k}\left(S_{k,n+1}^{(a,b)}+S_{k,n}^{(a,b)}\right)-b-\frac{a+2b}{k}=0
			\end{equation*}
			and hence we get
			\begin{equation*}
			S_{k,n+1}^{(a,b)}+S_{k,n}^{(a,b)} = bk+2b+a,
			\end{equation*}
			which contradicts our assumption. \\
			\item[(ii)] If $\omega^{l} = -1,$ then it follows from (\ref{eq_omega_l}) that
			\begin{equation*}
			S_{k,n+1}^{(a,b)}-S_{k,n}^{(a,b)}=bk -2b + a,
			\end{equation*}
			which also contradicts our assumption.
		\end{enumerate}
	\end{enumerate}
This completes the proof of the theorem.
\end{proof}
By a similar argument, we can prove the following slightly weaker result:
\begin{theorem}\label{thm_matrix_invertible_2}
	Let $k \ne 0$ be a real number.
	\begin{enumerate}
		\item[(a)] For an even integer $n \geq 2,$ if $S_{k,n+1}^{(a,b)} - S_{k,n}^{(a,b)} \ne -bk+2b-a$, then the matrix $A^{(k,a,b)}_{n} $ is invertible.
		\item[(b)] For an odd integer $n \geq 1,$ the matrix $B^{(k,a,b)}_{n}$ is invertible if and only if $S_{k,n+1}^{(a,b)}+S_{k,n}^{(a,b)} \ne bk+2b+a$.
	\end{enumerate}
\end{theorem}
\section{Singular $k$-FL Pairs of Various Type}\label{sec_pair}
Let $k,a$, and $b$ be real numbers, and let $\{S_{k,n}^{(a,b)}\}_{n \in \Z_{\geq 0}}$ be a generalized $k$-FL sequence. In this section, we will introduce a special kind of pairs of real numbers associated to the sequence $\{S_{k,n}^{(a,b)}\}_{n \in \Z_{\geq 0}}$. To this aim, we start with the following two lemmas:
\begin{lemma}\label{lem_fn_gn}
	For each integer $n \geq 1$, there exist (unique) $f_{n}, g_{n} \in \Z[T]$ such that $S^{(a,b)}_{k,n}=f_{n}(k)a+g_{n}(k)b$.
\end{lemma}
\begin{proof}
	We will prove the lemma by induction on $n$.
	\begin{enumerate}
		\item[(i)] $n=1$; by definition, we can take $f_{1}(T)=1, g_{1}(T)=T$.
		\item[(ii)] $n > 1$; suppose that there exist $f_{m} ,g_{m} \in \Z[T]$ such that $S^{(a,b)}_{k,m}=f_{m}(k)a+g_{m}(k)b$ for each $m \leq n-1$. By the recursion formula, we have
		\begin{align*}
		S^{(a,b)}_{k,n}&=k \cdot S^{(a,b)}_{k,n-1}+S^{(a,b)}_{k,n-2}\\
		&=k(f_{n-1}(k)a+g_{n-1}(k)b)+(f_{n-2}(k)a+g_{n-2}(k)b)\\
		&=(k f_{n-1}(k)+f_{n-2}(k))a + (kg_{n-1}(k)+g_{n-2}(k))b
		\end{align*}
		so that we can take
		\begin{equation*}
		f_{n}(T)=T \cdot f_{n-1}(T)+f_{n-2}(T)~~\textrm{and}~~ g_{n} (T)=T \cdot g_{n-1}(T)+g_{n-2}(T).
		\end{equation*}
	\end{enumerate}
	This completes the proof.
\end{proof}
For each integer $n \geq 1$, by virtue of Lemma \ref{lem_fn_gn}, we can write
\begin{equation}\label{eq_FG_1}
S_{k,n+1}^{(a,b)}-S_{k,n}^{(a,b)}=F_{n}(k)a+G_{n}(k)b
\end{equation}
and
\begin{equation}\label{eq_FG_2}
S_{k,n+1}^{(a,b)}+S_{k,n}^{(a,b)}=P_{n}(k)a+Q_{n}(k)b
\end{equation}
for some (unique) $F_{n},G_{n},P_{n},Q_{n} \in \Z[T]$. \\

We may find the degrees of those polynomials explicitly:
\begin{lemma}\label{lem_deg_poly}
	For each integer $n \geq 1,$ we have $\deg F_{n}=\deg P_{n} = n$ and $\deg G_{n} =\deg Q_{n}= n+1$.
\end{lemma}
\begin{proof}
	By construction, we have
	\begin{equation*}
	F_{n}(k)=F_{k,n+1}-F_{k,n}, ~~\textrm{and}~~G_{n}(k)=L_{k,n+1}-L_{k,n}
	\end{equation*}
	for each $k \in \R.$
	Similarly, we also have
	\begin{equation*}
	P_{n}(k)=F_{k,n+1}+F_{k,n}, ~~\textrm{and}~~Q_{n}(k)=L_{k,n+1}+L_{k,n}
	\end{equation*}
	for each $k \in \R$. According to Remark \ref{rm_binet} (see Appendix below), the Binet's formulas for $\{F_{k,n}\}_{n \in \Z_{\geq 0}}$ and $\{L_{k,n}\}_{n \in \Z_{\geq 0}}$
	are given by
	\begin{equation*}
	F_{k,n}=\frac{\alpha^{n} -\beta^{n}}{\alpha - \beta}~~\textrm{and}~~L_{k,n}=\alpha^{n} + \beta^{n}~~\textrm{for each}~n \geq 0
	\end{equation*}
	where $\alpha > \beta$ are the roots of the quadratic equation $x^2-kx-1=0$ so that $\alpha + \beta = k, \alpha \beta=-1$, and hence, it follows from the Binomial Theorem that
	\begin{equation*}
	F_{k,n}=k^{n-1}+\textrm{(lower terms in $k$)}~~\textrm{and}~~L_{k,n}=k^{n} + \textrm{(lower terms in $k$)}.
	\end{equation*}
	Hence we have
	\begin{align*}
	F_{n}(k) & =F_{k,n+1}-F_{k,n} =k^{n} + \textrm{(lower terms in $k$)},\\
	P_{n}(k) & =F_{k,n+1}+F_{k,n} =k^{n} + \textrm{(lower terms in $k$)},\\
	G_{n}(k) & =L_{k,n+1}-L_{k,n}=k^{n+1}+\textrm{(lower terms in $k$)},
	\end{align*}
	and
	\begin{equation*}
	Q_{n}(k) =L_{k,n+1}+L_{k,n}=k^{n+1}+\textrm{(lower terms in $k$)}.
	\end{equation*}
	Then since these equations hold for any $k \in \R$, and since $F_{n}, G_{n},P_{n},Q_{n} \in \Z[T]$, it follows that $\deg F_{n} =\deg P_{n} =n$ and $\deg G_{n} =\deg Q_{n} = n+1$, as desired.
\end{proof}
\begin{example}\label{ex_FGPQ}
	The following table gives the first four $F_{n}$'s, $G_{n}$'s, $P_{n}$'s, and $Q_{n}$'s: 
	\begin{center}
	\begin{tabular}{|c|c|c|c|c|}
		\hline
		$n$ & 0 & 1 &2 &3 \\
		%\multicolumn{2}{|c|}{Ene}\\
		\hline
		$F_{n}(k)$ & $1$ & $k-1$ & $k^2-k+1$& $k^3-k^2+2k-1$  \\
		\hline
		$G_{n}(k)$ & $k-2$ & $k^2-k+2$ & $k^3-k^2+3k-2$& $k^4-k^3+4k^2-3k+2$ \\
		\hline
		$P_{n}(k)$ & $ 1 $ & $ k+1  $ & $ k^2+k+1  $& $ k^3+k^2+2k+1  $ \\
		\hline
		$Q_{n}(k)$ & $ k+2 $ & $ k^2+k+2  $ & $ k^3+k^2+3k+2  $& $ k^4+k^3+4k^2+3k+2  $ \\
		\hline
	\end{tabular}
	\end{center}
\end{example}
Now, for each integer $n \geq 1$, we define three subsets of $\R$ as follows:
\begin{align*}
A_{n} & =\{k \in \R~|~F_{n} (k)+1=0 ~\textrm{or}~G_{n} (k)+k-2=0 \},\\
B_{n} & =\{k \in \R~|~F_{n} (k)-1=0 ~\textrm{or}~G_{n} (k)-k+2=0 \}
\end{align*}
and
\begin{equation*}
C_{n} = \{k \in \R~|~P_{n}(k)-1=0~\textrm{or}~Q_{n}(k)-k-2=0\}.
\end{equation*}
Clearly, $A_{n}, B_{n}, C_{n}$ are all finite by Lemma \ref{lem_deg_poly} for each $n \geq 1$.
\begin{example}[$A_{3}, B_{3}, C_{3}$]
	We first claim that $|A_{3}|=2$. Indeed, we have
	\begin{equation*}
	A_{3} =\{k \in \R~|~F_{3}(k)+1=0~\textrm{or}~G_{3}(k)+k-2=0 \}.
	\end{equation*}
	
	From Example \ref{ex_FGPQ}, we know that $F_{3}(k)+1=k^3-k^2+2k=k(k^2-k+2)=0$, and hence $0 \in A_{3}$. Also, we have $G_{3}(k)+k-2=k^4-k^3+4k^2-2k=k(k^3-k^2+4k-2)=0$. Since there is a unique real solution for the equation $k^3-k^2+4k-2=0,$ it follows that $|A_{3}|=2.$ In particular, $A_{3}$ is finite. \\
	Next, we claim that $B_{3}=\{1\}.$ Indeed, we have
	\begin{equation*}
	B_{3} =\{k \in \R~|~F_{3}(k)-1=0~\textrm{or}~G_{3}(k)-k+2=0 \}.
	\end{equation*}
	From Example \ref{ex_FGPQ}, we know that $F_{3}(k)-1=k^3-k^2+2k-2=(k-1)(k^2+2)=0$, and hence $1 \in B_{3}$. Also, we have $G_{3}(k)-k+2=k^4-k^3+4k^2-4k+4 >0$ for all $k \in \R$. Thus, we have $B_{3} =\{1\}$. In particular, $B_{3}$ is finite. \\
	Finally, we claim that $|C_{3}|=2$. Indeed, we have
	\begin{equation*}
	C_{3} =\{k \in \R~|~P_{3}(k)-1=0~\textrm{or}~Q_{3}(k)-k-2=0 \}.
	\end{equation*}
	From Example \ref{ex_FGPQ}, we know that $P_{3}(k)-1=k^3+k^2+2k=k(k^2+k+2)=0$, and hence $0 \in C_{3}$. Also, we have $G_{3}(k)-k-2=k^4+k^3+4k^2+2k=k(k^3+k^2+4k+2)=0$. Since there is a unique real solution for the equation $k^3+k^2+4k+2=0$, it follows that $|C_{3}|=2$. In particular, $C_{3}$ is finite.
\end{example}
Now, we are interested in solving special kind of equations related to a generalized $k$-FL sequence:
\begin{lemma}\label{lem_uniquepair_1}
	Let $n \geq 1$ be an integer. Then there is a unique pair $(a,b) \ne (0,0)$ (in terms of $k$) satisfying both of the equations $S^{(a,b)}_{k,n+1}-S^{(a,b)}_{k,n}=-bk+2b-a$ and $4a^3+27b^2=0$ for each $k \not \in A_{n}$.
\end{lemma}
\begin{proof}
	By the relation (\ref{eq_FG_1}) above, the first equation becomes
	\begin{equation}\label{eq_FG_3}
	(F_{n}(k)+1)a + (G_{n}(k)+k-2)b = 0.
	\end{equation}
	Assume $(a,b)\ne (0,0).$ Since $k \not \in A_{n}$, we have $F_{n}(k)+1 \ne 0$ and $G_{n}(k)+k-2 \ne 0$. Then from the equation (\ref{eq_FG_3}), we get $a=-\frac{G_{n}(k)+k-2}{F_{n}(k)+1}b$. After equating with the second equation, we can easily see that there is a unique solution (other than $(0,0)$) in terms of $k$, namely, $a=-\frac{27(F_{n}(k)+1)^{2}}{4(G_{n}(k)+k-2)^{2}}$, $b=\frac{27(F_{n}(k)+1)^{3}}{4(G_{n}(k)+k-2)^{3}}$.
\end{proof}
\begin{lemma}\label{lem_uniquepair_2}
	Let $n \geq 1$ be an integer. Then there is a unique pair $(a,b) \ne (0,0)$ (in terms of $k$) satisfying both of the equations $S^{(a,b)}_{k,n+1}-S^{(a,b)}_{k,n}=bk-2b+a$ and $4a^3+27b^2=0$ for each $k \not \in B_{n}$.
\end{lemma}
\begin{proof}
	By the relation (\ref{eq_FG_1}) above, the first equation becomes
	\begin{equation}\label{eq_FG_4}
	(F_{n}(k)-1)a + (G_{n}(k)-k+2)b = 0.
	\end{equation}
	Assume $(a,b)\ne (0,0)$. Since $k \not \in B_{n}$, we have $F_{n}(k)-1 \ne 0$ and $G_{n}(k)-k+2 \ne 0$. Then from the equation (\ref{eq_FG_4}), we get $a=-\frac{G_{n}(k)-k+2}{F_{n}(k)-1}b$. After equating with the second equation, we can easily see that there is a unique solution (other than $(0,0)$) in terms of $k$, namely, $a=-\frac{27(F_{n}(k)-1)^{2}}{4(G_{n}(k)-k+2)^{2}}$, $b=\frac{27(F_{n}(k)-1)^{3}}{4(G_{n}(k)-k+2)^{3}}.$
\end{proof}
Finally, we also have
\begin{lemma}\label{lem_uniquepair_3}
	Let $n \geq 1$ be an integer. Then there is a unique pair $(a,b) \ne (0,0)$ (in terms of $k$) satisfying both of the equations $S^{(a,b)}_{k,n+1}+S^{(a,b)}_{k,n}=bk+2b+a$ and $4a^3+27b^2=0$ for each $k \not \in C_{n}$.
\end{lemma}
\begin{proof}
	By the relation (\ref{eq_FG_2}) above, the first equation becomes
	\begin{equation}\label{eq_PQ_1}
	(P_{n}(k)-1)a + (Q_{n}(k)-k-2)b = 0.
	\end{equation}
	Assume $(a,b)\ne (0,0)$. Since $k \not \in C_{n}$, we have $P_{n}(k)-1 \ne 0$ and $Q_{n}(k)-k-2 \ne 0$. Then from the equation (\ref{eq_PQ_1}), we get $a=-\frac{Q_{n}(k)-k-2}{P_{n}(k)-1}b$. After equating with the second equation, we can easily see that there is a unique solution (other than $(0,0)$) in terms of $k$, namely, $a=-\frac{27(P_{n}(k)-1)^{2}}{4(Q_{n}(k)-k-2)^{2}}$, $b=\frac{27(P_{n}(k)-1)^{3}}{4(Q_{n}(k)-k-2)^{3}}$.
\end{proof}
These lemmas lead to the following three new concepts:
\begin{definition}
	Let $n \geq 1$ be an integer. For each $k \not \in A_{n}$, we define the \emph{singular $k$-FL pair of level $n$ of Type 1} to be the unique pair $(a,b) \ne (0,0)$ of real numbers of the form given as in Lemma \ref{lem_uniquepair_1}. We will denote the singular $k$-FL pair of level $n$ of Type 1 by $(a_{k,1}^{(n)},b_{k,1}^{(n)}).$
\end{definition}
Similarly,
\begin{definition}
	Let $n \geq 1$ be an integer. For each $k \not \in B_{n}$, we define the \emph{singular $k$-FL pair of level $n$ of Type 2} to be the unique pair $(a,b) \ne (0,0)$ of real numbers of the form given as in Lemma \ref{lem_uniquepair_2}. We will denote the singular $k$-FL pair of level $n$ of Type 2 by $(a_{k,2}^{(n)},b_{k,2}^{(n)}).$
\end{definition}
Finally,
\begin{definition}
	Let $n \geq 1$ be an integer. For each $k \not \in C_{n}$, we define the \emph{singular $k$-FL pair of level $n$ of Type 3} to be the unique pair $(a,b) \ne (0,0)$ of real numbers of the form given as in Lemma \ref{lem_uniquepair_3}. We will denote the singular $k$-FL pair of level $n$ of Type 3 by $(a_{k,3}^{(n)},b_{k,3}^{(n)}).$
\end{definition}
Here comes an example of each kind:
\begin{example}[Singular $k$-FL pair of level $2$ of Type 1]
	Let $n=2$. From Example \ref{ex_FGPQ}, we have $F_{2}(k)=k^2-k+1, G_{2}(k)=k^{3}-k^{2}+3k-2$ so that $A_{2} =\{1\}.$ Now, for each $k \in \R \setminus \{1\},$ the singular $k$-FL pair of level $2$ of Type 1 is given by $(a^{(2)}_{k,1},b^{(2)}_{k,1})$ where
	\begin{equation*}
	a^{(2)}_{k,1}= -\frac{27(k^2-k+2)^{2}}{4(k^3-k^2+4k-4)^{2}}~~\textrm{and}~~ b^{(2)}_{k,1}=\frac{27(k^2-k+2)^{3}}{4(k^3-k^2+4k-4)^{3}}.
	\end{equation*}
\end{example}
Similarly,
\begin{example}[Singular $k$-FL pair of level $2$ of Type 2]
	Let $n=2$. From Example \ref{ex_FGPQ}, we have $F_{2}(k)=k^2-k+1, G_{2}(k)=k^{3}-k^{2}+3k-2$ so that $B_{2} =\{0,1\}.$ Now, for each $k \in \R \setminus \{0,1\},$ the singular $k$-FL pair of level $2$ of Type 2 is given by $(a^{(2)}_{k,2},b^{(2)}_{k,2})$ where
	\begin{equation*}
	a^{(2)}_{k,2}= -\frac{27(k-1)^{2}}{4(k^2-k+2)^{2}}~~\textrm{and}~~ b^{(2)}_{k,2}=\frac{27(k-1)^{3}}{4(k^2-k+2)^{3}}.
	\end{equation*}
\end{example}
Finally,
\begin{example}[Singular $k$-FL pair of level $2$ of Type 3]
	Let $n=2$. From Example \ref{ex_FGPQ}, we have $P_{2}(k)=k^2+k+1, Q_{2}(k)=k^{3}+k^{2}+3k+2$ so that $C_{2} =\{-1,0\}.$ Now, for each $k \in \R \setminus \{-1,0\},$ the singular $k$-FL pair of level $2$ of Type 3 is given by $(a^{(2)}_{k,3},b^{(2)}_{k,3})$ where
	\begin{equation*}
	a^{(2)}_{k,3}= -\frac{27(k+1)^{2}}{4(k^2+k+2)^{2}}~~\textrm{and}~~ b^{(2)}_{k,3}=\frac{27(k+1)^{3}}{4(k^2+k+2)^{3}}.
	\end{equation*}
\end{example}
Now, fix an integer $n \geq 1$. For $k \not \in B_{n}$, let $(a,b)=(a_{k,2}^{(n)},b_{k,2}^{(n)})$ be a singular $k$-FL pair of level $n$ of Type 2. Associated to $(a,b)$, we introduce a singular plane cubic curve as follows: let $E_{(a,b)}$ be a curve defined by the equation $y^{2} =x^{3}+ax+b.$ Since $(a,b)$ satisfies the equation $4a^{3} +27b^{2} =0$, $E_{(a,b)}$ is singular. Note that $E_{(a,b)}$ is not necessarily defined over $\Q$. \\
Motivated by the above observation, let
\begin{equation*}
S_2^{(n)}=\{k \in \R \setminus B_n~|~E_{(a,b)}~\textrm{is defined over}~\Q \}
\end{equation*}
for each integer $n \geq 1$. Then since $\Q \setminus B_{n} \subseteq S_{2}^{(n)},$ we can see that $S_{2}^{(n)}$ is infinite. On the other hand, the set $S_{2}^{(n)}$ turns out to be not too big in the sense of Theorem \ref{thm_s2_countable} below. Before stating the theorem, we need a lemma:
\begin{lemma}\label{lem_FG_equiv}
	Let $F,G \in \Z[T]$ be two polynomials. For any real number $k$ with $F(k)-1 \ne 0$ and $G(k)-k+2 \ne 0,$ let
	\begin{equation*}
	a(k)=-\frac{27(F(k)-1)^{2}}{4(G(k)-k+2)^{2}}~~\textrm{and}~~b(k)=\frac{27(F(k)-1)^{3}}{4(G(k)-k+2)^{3}}.
	\end{equation*}
	Then the following are equivalent:
	\begin{enumerate}
		\item[(i)] $a(k),b(k) \in \Q$;
		\item[(ii)] $\frac{(F(k)-1)^{2}}{(G(k)-k+2)^{2}}, \frac{(F(k)-1)^{3}}{(G(k)-k+2)^{3}} \in \Q$;
		\item[(iii)] $\frac{F(k)-1}{G(k)-k+2} \in \Q.$
	\end{enumerate}
\end{lemma}
\begin{proof}
	$(i) \Rightarrow (ii)$ follows from the fact that $\frac{27}{4} \in \Q.$ \\
	$(ii) \Rightarrow (iii)$ follows from the fact that $\frac{F(k)-1}{G(k)-k+2}=\frac{\frac{(F(k)-1)^{3}}{(G(k)-k+2)^{3}}}{\frac{(F(k)-1)^{2}}{(G(k)-k+2)^{2}}}$ if $F(k)-1 \ne 0$ and $G(k)-k+2 \ne 0$.  \\
	$(iii) \Rightarrow (i)$ Trivial.
\end{proof}
Now, we prove the following
\begin{theorem}\label{thm_s2_countable}
	$S_2^{(n)}$ is countable for any integer $n \geq 1$.
\end{theorem}
\begin{proof} Fix an integer $n \geq 1.$ For each $k \not \in B_n,$ as in the proof of Lemma \ref{lem_uniquepair_2}, we can write
	\begin{equation*}
	a= a_{k,2}^{(n)}=-\frac{27(F_n(k)-1)^2}{4(G_n(k)-k+2)^2}~~\textrm{and}~~ b=b_{k,2}^{(n)}=\frac{27(F_n(k)-1)^3}{4(G_n(k)-k+2)^3}
	\end{equation*}
	for some $F_n,G_n \in \Z[T].$ Now, $E_{(a,b)}$ is defined over $\Q$ if and only if $a,b \in \Q.$
	By Lemma \ref{lem_FG_equiv}, it suffices to find $k \in \R \setminus B_n$ such that $\frac{F_n(k)-1}{G_n(k)-k+2} \in \Q.$ Assume that $\frac{F_n(k)-1}{G_n(k)-k+2} = \frac{q}{p}$ for some $p,q \in \Z$ with $\gcd(p,q)=1.$ By equating, we have
	\begin{equation}\label{eq_FG_5}
	qG_n(k)-pF_n(k)-qk+2q+p=0.
	\end{equation}
	By Lemma \ref{lem_deg_poly}, for each $(p,q) \in \Z^2$ with $\gcd(p,q)=1$, there are at most $n+1$ real roots of (\ref{eq_FG_5}) in $k$, and there are only countably many such pairs $(p,q)$ (since $\Z^2$ is countable). Hence it follows that $S_2^{(n)}$ is countable.
\end{proof}
We may have similar results for singular $k$-FL pairs of level $n$ of Type 1 or Type 3, too.

\section{Applications}\label{sec_application}

Our first application gives some information about the cardinality of the Zsigmondy set of integer sequences in view of Introduction. \\

\textbf{Application 1}~~Let $\mathcal{A}=\{a_n \}_{n \in \Z_{\geq 0}}$ be an integer sequence. A prime number $p$ is called a \emph{primitive prime divisor for} $a_n$ if $p~|~a_n$, but $p \nmid a_m$ for all $0 \leq m < n$ with $a_m \ne 0$ \cite{Silverman2013}. The \emph{Zsigmondy set} of $\mathcal{A}$ is defined to be
\begin{equation*}
\mathcal{Z}(\mathcal{A})=\{n \geq 0~|~a_n~\textrm{has no primitive prime divisors}\}.
\end{equation*}

It seems to be an interesting question to ask whether the Zsigmondy set of a given sequence $\mathcal{A}=\{a_n\}_{n \in \Z_{\geq 0}}$ is finite or not. In this direction, there are several results in the literature, one of which, we introduce here for our later use.

\begin{example}[Carmichael\cite{Carmichael1913}]\label{exa_Carmichael_1}
	Let $\mathcal{A}=\{F_n\}_{n \in \Z_{\geq 0}}$ be the classical Fibonacci sequence. Then $\mathcal{Z}(\mathcal{A})=\{1,2,6,12\}.$ In particular, $\mathcal{Z}(\mathcal{A})$ is bounded.
\end{example}

Adopting one of the idea given in Introduction, we generalize this concept to a family of integer sequences in the following sense: let $\mathcal{F}$ be a family of integer sequences. We let
\begin{equation*}
N_{\mathcal{F}}=\sup_{\mathcal{A}:=\{a_n\}_{n \in \Z_{\geq 0}} \in \mathcal{F}} |\mathcal{Z}(\mathcal{A})|
\end{equation*}
where $|S|$ denotes the cardinality of a set $S.$ We can ask whether $N_{\mathcal{F}}$ is finite or not. A classical result of Zsigmondy gives an example of $N_{\mathcal{F}}$ being finite.

\begin{example}(Zsigmondy\cite{Zsigmondy1892})
	Let $\mathcal{F}$ be a family of the integer sequences of the form $\{u^n - v^n\}_{n \in \Z_{\geq 0}}$ where $u>v>0$ are relatively prime integers. Then $N_{\mathcal{F}} = 3.$ In particular, $N_{\mathcal{F}}$ is bounded.
\end{example}

Also, Example \ref{exa_Carmichael_1} can be restated as the following

\begin{example}
	Let $\mathcal{F}$ be the family that consists of the classical Fibonacci sequence. In other words, we let $\mathcal{F}=\{\{F_n\}_{n \in \Z_{\geq 0}}\}$ where $F_n$ denotes the $n$-th Fibonacci number. Then $N_{\mathcal{F}}=4.$
\end{example}

Before answering Question \ref{que_Zsigmondy_1} of Introduction, we need to look at the following classical result of Carmichael in more detail: let $P,Q$ be two nonzero integers, and put
\begin{equation*}
U_n(P,Q):=\frac{\alpha^n - \beta^n }{\alpha - \beta},~~~V_n (P,Q):=\alpha^n + \beta^n ~~~(n \geq 0)
\end{equation*}
where $\alpha, \beta$ are the zeros of $x^2 - Px +Q.$ Then we have the following

\begin{theorem}[Carmichael\cite{Carmichael1913}]\label{thm_Carmichael_2}
	(a) If $D>0$, then for all integers $n \geq 0$ with $n \not \in \{1,2,6\}$,  $U_n$ has a primitive prime divisor, except for the case when $n=12, P=\pm 1, Q=-1.$ \\
	(b) If $D$ is a perfect square, then for all integers $n \geq 0$, $U_n$ has a primitive prime divisor, except for the case when $n=6, P=\pm 3, Q=2.$ (This includes the Zsigmondy's theorem.)
\end{theorem}
As an immediate consequence of Theorem \ref{thm_Carmichael_2}, we get
\begin{corollary}\label{cor_k-fibo_1}
	Let $k \geq 2$ be an integer, and let $\mathcal{A}=\{F_{k,n}\}_{n \in \Z_{\geq 0}}$ be the $k$-Fibonacci sequence. Then we have
	\begin{equation*}
	\{1,2\} \subseteq \mathcal{Z}(\mathcal{A}) \subseteq \{1,2,6\}.
	\end{equation*}
\end{corollary}
\begin{example}
	Let $\mathcal{A}=\{F_{k,n}\}_{n \in \Z_{\geq 0}}$ be the $k$-Fibonacci sequence. If $k=2,$ then $\mathcal{Z}(\mathcal{A})=\{1,2\}$, and if $k=3,$ then $\mathcal{Z}(\mathcal{A})=\{1,2,6\}.$
\end{example}

Finally, to give a positive answer for Question \ref{que_Zsigmondy_1}, we will consider a family of generalized $k$-FL sequences. Namely, we have the following
\begin{theorem}\label{thm_k-FL_1}
	For any integers $k,N \geq 1$, there exist $a,b \in \R$ with $|\mathcal{Z}(\{S_{k,n}^{(a,b)}\}_{n \in \Z_{\geq 0}})|\geq N.$
\end{theorem}
\begin{proof}
	Recall that if $n~|~m$, then $F_{k,n}~|~F_{k,m}$ by \cite[Proposition 8]{Bolat2010}. We choose $r \geq 1$ large enough so that $\sigma_0 (r) \geq [N/2]+1$ (where $\sigma_0 (r)$ denotes the number of the divisors of $r$), and we solve the following equations
	\begin{equation*}
	2b=(-1)^{r+1}\cdot F_{k,r}, ~~bk+a=(-1)^{r}\cdot F_{k,r-1}
	\end{equation*}
	to get
	\begin{equation*}
	a=(-1)^{r}\cdot \left(F_{k,r-1} + \frac{k \cdot F_{k,r}}{2}\right),~~ b=(-1)^{r+1}\cdot \frac{F_{k,r}}{2}.
	\end{equation*}
	Using the values of $a$ and $b$ obtained above, we get a generalized $k$-FL sequence defined by
	\begin{displaymath}
	S_{k,n}^{(a,b)} = \left\{ \begin{array}{ll}
	(-1)^{r+1-n} \cdot F_{k,r-n} & \textrm{if $n \leq r$}\\
	F_{k,n-r} & \textrm{if $n > r$}.
	\end{array} \right.
	\end{displaymath}
	Then since $\{r-d~|~d \in \mathcal{P}(r) \} \cup \{r+d~|~d \in \mathcal{P}(r)\} \subset \mathcal{Z}(\{S_{k,n}^{(a,b)}\}_{n \in \Z_{\geq 0}})$ where $\mathcal{P}(r)$ denotes the set of positive divisors of $r$, we have
	\begin{equation*}
	|\mathcal{Z}(\{S_{k,n}^{(a,b)}\}_{n \in \Z_{\geq 0}})| \geq 2 \cdot \sigma_0 (r) \geq N
	\end{equation*}
	which completes the proof of the theorem.
\end{proof}

As a consequence, we can answer Question \ref{que_Zsigmondy_1}.
\begin{corollary}
	Let $k \geq 1$ be an integer. Let $\mathcal{F}$ be a family of generalized $k$-FL sequences of the form $\{S_{k,n}^{(a(N),b(N))}\}_{n \in \Z_{\geq 0}}$ where $N$ runs over all positive integers, and $a(N),b(N) $ are two real numbers that are obtained as in Theorem \ref{thm_k-FL_1}. Then $N_{\mathcal{F}}$ is infinite, and for each $\mathcal{A} \in \mathcal{F},$ $\mathcal{Z}(\mathcal{A})$ is finite.
\end{corollary}
\begin{proof}
	The first statement follows from the theorem by letting $N \rightarrow \infty.$ The second statement follows from Example \ref{exa_Carmichael_1}  and Corollary \ref{cor_k-fibo_1}.
\end{proof}

As an application to algebraic geometry, we investigate the determinantal variety of the skew circulant matrix and the circulant matrix associated to a generalized $k$-FL sequence. \\

\textbf{Application 2}~~For real numbers $k,a$, and $b$, let $\{S_{k,n}^{(a,b)}\}_{n \in \Z_{\geq 0}}$ be a generalized $k$-FL sequence, and let $E_{(a,b)}$ be the curve defined by the equation $y^2 =x^3+ax+b.$ We start with the following motivating
\begin{example}\label{ex_A3B3}
	Assume that $k \ne 0$ and let
	\begin{equation*}
	A^{(k,a,b)}_3 =\left(\begin{array}{ccc} S^{(a,b)}_{k,1} & S^{(a,b)}_{k,2} & S^{(a,b)}_{k,3} \\ -S^{(a,b)}_{k,3} & S^{(a,b)}_{k,1} & S^{(a,b)}_{k,2} \\ -S^{(a,b)}_{k,2} & -S^{(a,b)}_{k,3} & S^{(a,b)}_{k,1}
	\end{array}\right)
	\end{equation*}
	and
	\begin{equation*}
	B^{(k,a,b)}_3 =\left(\begin{array}{ccc} S^{(a,b)}_{k,1} & S^{(a,b)}_{k,2} & S^{(a,b)}_{k,3} \\ S^{(a,b)}_{k,3} & S^{(a,b)}_{k,1} & S^{(a,b)}_{k,2} \\ S^{(a,b)}_{k,2} & S^{(a,b)}_{k,3} & S^{(a,b)}_{k,1}
	\end{array}\right)
	\end{equation*}
	be two matrices associated to $\{S_{k,n}^{(a,b)}\}_{\Z_{\geq 0}}$. For $r=1,2$, we define the following two sets
	\begin{align*}
	S^r_{3}= & \{(k,a,b) \in (\R \setminus A_3) \times \R \times \R~|~\textrm{rank}(A_3^{(k,a,b)}) \leq r, E_{(a,b)}~\textrm{is singular}\} \\
	& \setminus \{(k,0,0)~|~k \in \R\}
	\end{align*}
	and
	\begin{align*}
	T^r_{3}= & \{(k,a,b) \in (\R \setminus C_3) \times \R \times \R~|~\textrm{rank}(B_3^{(k,a,b)}) \leq r, E_{(a,b)}~\textrm{is singular}\} \\
	& \setminus \{(k,0,0)~|~k \in \R\}.
	\end{align*}
	Then by Theorem \ref{thm_matrix_invertible_1}-(a), we have
	\begin{equation*}
	S_3^2 = \{(k,a_{k,1}^{(3)},b_{k,1}^{(3)})~|~k \in \R \setminus A_3 \}.
	\end{equation*}
	Similarly, by Theorem \ref{thm_matrix_invertible_2}-(b), we have
	\begin{equation*}
	T_3^2 = \{(k,a_{k,3}^{(3)},b_{k,3}^{(3)})~|~k \in \R \setminus C_3 \}.
	\end{equation*}
	Since $A_3$ and $C_3$ are finite, we see that $S_3^2$ and $T_3^2$ are uncountable.
	What can we say about $S_3^1$ and $T_3^1$? The next result gives one possible answer:
	\begin{lemma}\label{lem_S3_T3_finite}
		$S_3^1$ and $T_3^1$ are finite.
	\end{lemma}
	\begin{proof}
		Note that $\textrm{rank}(A_3^{(k,a,b)}) \leq 1$ if and only if all the $2 \times 2$ minors of $A_3^{(k,a,b)}$ are 0. Also, we have $S_3^1 \subseteq S_3^2.$ Thus, any element of $S_3^1$ is of the form $(k,a^{(3)}_{k,1},b^{(3)}_{k,1})$ for some $k \in \R \setminus A_3$ where
		$a^{(3)}_{k,1}=-\frac{27(F_3(k)+1)^2}{4(G_3(k)+k-2)^2}~~\textrm{and}~~ b^{(3)}_{k,1}=\frac{27(F_3(k)+1)^3}{4(G_3(k)+k-2)^3}$ for some $F_3,G_3 \in \Z[T].$ \\
		Let $(k,a^{(3)}_{k,1},b^{(3)}_{k,1}) \in S_3^1$ so that $\textrm{rank}(A_3^{(k,a^{(3)}_{k,1},b^{(3)}_{k,1})})\leq 1 .$ By the above observation, it follows that we have the following system of equations
		\begin{align}
		\left(S_{k,1}^{(a^{(3)}_{k,1},b^{(3)}_{k,1})} \right)^2 +S_{k,2}^{(a^{(3)}_{k,1},b^{(3)}_{k,1})} \cdot S_{k,3}^{(a^{(3)}_{k,1},b^{(3)}_{k,1})} & =0 \nonumber\\
		\left(S_{k,2}^{(a^{(3)}_{k,1},b^{(3)}_{k,1})} \right)^2 -S_{k,3}^{(a^{(3)}_{k,1},b^{(3)}_{k,1})} \cdot S_{k,1}^{(a^{(3)}_{k,1},b^{(3)}_{k,1})} & =0, \nonumber\\
		\textrm{and}~\left(S_{k,3}^{(a^{(3)}_{k,1},b^{(3)}_{k,1})} \right)^2 +S_{k,1}^{(a^{(3)}_{k,1},b^{(3)}_{k,1})} \cdot S_{k,2}^{(a^{(3)}_{k,1},b^{(3)}_{k,1})}& =0.\label{eq_application_1}
		\end{align}
		Since $S_{k,1}^{(a^{(3)}_{k,1},b^{(3)}_{k,1})}=b^{(3)}_{k,1} k+a^{(3)}_{k,1}, S_{k,2}^{(a^{(3)}_{k,1},b^{(3)}_{k,1})}=b^{(3)}_{k,1} k^2 +a^{(3)}_{k,1} k+2b^{(3)}_{k,1},$ and $S_{k,3}^{(a^{(3)}_{k,1},b^{(3)}_{k,1})}=b^{(3)}_{k,1} k^3 +a^{(3)}_{k,1} k^2 +3b^{(3)}_{k,1} k+a^{(3)}_{k,1},$ and by using the fact that $a^{(3)}_{k,1}=-\frac{27(F_3(k)+1)^2}{4(G_3(k)+k-2)^2}$ and $b^{(3)}_{k,1}=\frac{27(F_3(k)-1)^3}{4(G_3(k)-k+2)^3}$, we can see that all the three equations of (\ref{eq_application_1}) have finitely many real solutions in $k$, and this proves the fact that $S_3^1$ is finite. \\
		Similarly, note that $\textrm{rank}(B_3^{(k,a,b)}) \leq 1$ if and only if all the $2 \times 2$ minors of $B_3^{(k,a,b)}$ are 0. Also, we have $T_3^1 \subseteq T_3^2$. Thus, any element of $T_3^1$ is of the form $(k,a^{(3)}_{k,3},b^{(3)}_{k,3})$ for some $k \in \R \setminus C_3$ where $a^{(3)}_{k,3}=-\frac{27(P_3(k)-1)^2}{4(Q_3(k)-k-2)^2}$ and $b^{(3)}_{k,3}=\frac{27(P_3(k)-1)^3}{4(Q_3(k)-k-2)^3}$ for some $P_3,Q_3 \in \Z[T]$. \\
		Let $(k,a^{(3)}_{k,3},b^{(3)}_{k,3}) \in T_3^1$ so that $\textrm{rank}(B_3^{(k,a^{(3)}_{k,3},b^{(3)}_{k,3})})\leq 1 .$ By the above observation, it follows that we have the following system of equations
		\begin{align}
		\left(S_{k,1}^{(a^{(3)}_{k,3},b^{(3)}_{k,3})} \right)^2 -S_{k,2}^{(a^{(3)}_{k,3},b^{(3)}_{k,3})} \cdot S_{k,3}^{(a^{(3)}_{k,3},b^{(3)}_{k,3})} & =0, \nonumber \\\left(S_{k,2}^{(a^{(3)}_{k,3},b^{(3)}_{k,3})} \right)^2 -S_{k,3}^{(a^{(3)}_{k,3},b^{(3)}_{k,3})} \cdot S_{k,1}^{(a^{(3)}_{k,3},b^{(3)}_{k,3})} & =0, \nonumber \\
		\textrm{and}~\left(S_{k,3}^{(a^{(3)}_{k,3},b^{(3)}_{k,3})} \right)^2 -S_{k,1}^{(a^{(3)}_{k,3},b^{(3)}_{k,3})} \cdot S_{k,2}^{(a^{(3)}_{k,3},b^{(3)}_{k,3})} & =0.\label{eq_application_2}
		\end{align}
		Since $S_{k,1}^{(a^{(3)}_{k,3},b^{(3)}_{k,3})}=b^{(3)}_{k,3} k+a^{(3)}_{k,3}, S_{k,2}^{(a^{(3)}_{k,3},b^{(3)}_{k,3})}=b^{(3)}_{k,3} k^2 +a^{(3)}_{k,3} k+2b^{(3)}_{k,3},$ and $S_{k,3}^{(a^{(3)}_{k,3},b^{(3)}_{k,3})}=b^{(3)}_{k,3} k^3 +a^{(3)}_{k,3} k^2 +3b^{(3)}_{k,3} k+a^{(3)}_{k,3},$ and by using the fact that $a^{(3)}_{k,3}=-\frac{27(P_3(k)-1)^2}{4(Q_3(k)-k-2)^2}$ and $b^{(3)}_{k,3}=\frac{27(P_3(k)-1)^3}{4(Q_3(k)-k-2)^3}$, we can see that all the three equations of (\ref{eq_application_2}) have finitely many real solutions in $k$. Consequently, we can conclude that $T_3^1$ is also finite.
	\end{proof}
	As an immediate consequence of the lemma, a matrix $A_3^{(k,a,b)}$ (resp. $B_3^{(k,a,b)}$) has rank exactly $2$ for all but finitely many triples $(k,a,b)\in S_3^2$ (resp. $(k,a,b)\in T_3^2$).
\end{example}
Now, Example \ref{ex_A3B3} leads us to the following
\begin{theorem}\label{thm_matrix_sn_skewcirculant}
	For $k \ne 0$ and an odd integer $n \geq 3,$ consider the associated skew circulant matrix
	\begin{equation*}
	A^{(k,a,b)}_n :=\left(\begin{array}{cccc} S^{(a,b)}_{k,1} & S^{(a,b)}_{k,2} & \cdots & S^{(a,b)}_{k,n} \\ -S^{(a,b)}_{k,n} & S^{(a,b)}_{k,1} &\cdots & S^{(a,b)}_{k,n-1} \\ \vdots & \vdots & \vdots & \vdots \\ -S^{(a,b)}_{k,2} & -S^{(a,b)}_{k,3} & \cdots & S^{(a,b)}_{k,1}
	\end{array}\right).
	\end{equation*}
	For $1 \leq r \leq n-1,$ let $Y_r =\{M \in M_{n \times n}(\R)~|~\textrm{rank}(M) \leq r \}$ and let
	\begin{align*}
	S_n^r = & \{(k,a,b)\in (\R \setminus A_n)\times \R \times \R~|~A_n^{(k,a,b)} \in Y_r, E_{(a,b)}~\textrm{is singular}\} \\
	& \setminus \{(k,0,0)~|~k \in \R \}.
	\end{align*}
	Then we have:
	\begin{enumerate}
		\item[(a)] $S_n^{n-1}=\{(k,a_{k,1}^{(n)},b_{k,1}^{(n)})~|~k \in \R \setminus A_n \}$.
		\item[(b)] A matrix $A_n^{(k,a,b)}$ has rank exactly $n-1$ for all but finitely many triples $(k,a,b)\in S_n^{n-1}$.
		\item[(c)] For each $1 \leq r \leq n-2,$ a matrix $A_n^{(k,a,b)}$ has rank $r$ for only finitely many triples $(k,a,b)\in S_n^{n-1}$.
	\end{enumerate}
\end{theorem}
\begin{proof}
	\begin{enumerate}
		\item[(a)] Note that $A_n^{(k,a,b)} \in Y_{n-1}$ if and only if $A_n^{(k,a,b)}$ is not invertible. Hence the result follows from Theorem \ref{thm_matrix_invertible_1}-(a), and the definition of $a_{k,1}^{(n)}$ and $b_{k,1}^{(n)}$.
		\item[(b)] Note that $A_n^{(k,a,b)} \in Y_{n-2}$ if and only if all the $(n-1) \times (n-1)$ minors of $A_n^{(k,a,b)}$ are 0. Also, we have $S_n^{n-2} \subseteq S_n^{n-1}.$ Thus, any element of $S_n^{n-2}$ is of the form $(k,a^{(n)}_{k,1}, b^{(n)}_{k,1})$ for some $k \in \R \setminus A_n$ where $a^{(n)}_{k,1}=-\frac{27(F_n(k)+1)^2}{4(G_n(k)+k-2)^2}~~\textrm{and}~~ b^{(n)}_{k,1}=\frac{27(F_n(k)+1)^3}{4(G_n(k)+k-2)^3}$ for some $F_n,G_n \in \Z[T].$ Let $(k,a^{(n)}_{k,1},b^{(n)}_{k,1}) \in S_n^{n-2}$ so that $A_n^{(k,a^{(n)}_{k,1},b^{(n)}_{k,1})} \in Y_{n-2}.$ By the above observation, it follows that we have a system of finitely many equations in $S_{k,1}^{(a^{(n)}_{k,1},b^{(n)}_{k,1})},\cdots,S_{k,n}^{(a^{(n)}_{k,1},b^{(n)}_{k,1})}$. Now, rewrite them in terms of $a_{k,1}^{(n)},b_{k,1}^{(n)}$ according to definition, and by using the fact that $a^{(n)}_{k,1}=-\frac{27(F_n(k)+1)^2}{4(G_n(k)+k-2)^2}$ and $b^{(n)}_{k,1}=\frac{27(F_n(k)+1)^3}{4(G_n(k)+k-2)^3}$, we can see that there are only finitely many real solutions in $k$ for the given system of equations. Hence $S_n^{n-2}$ is finite and $A_n^{(k,a,b)}$ has rank exactly $n-1$ for all but $\sharp S_n^{n-2}$-many triples $(k,a,b)\in S_n^{n-1}$.
		\item[(c)] In the proof of (b), we have seen that $S_n^{n-2}$ is finite. For each $1 \leq r \leq n-2,$ we have $S_n^{r} \subseteq S_n^{n-2}$ so that $S_n^r$ is finite, and hence $A_n^{(k,a,b)}$ has rank $r$ for only $(\sharp S_n^r - \sharp S_n^{r-1})$-many triples $(k,a,b) \in S_n^{n-1}$.
	\end{enumerate}
This completes the proof.
\end{proof}
\begin{remark}
	For $k \ne 0$ and an even integer $n \geq 2,$ we have slightly weaker version of Theorem \ref{thm_matrix_sn_skewcirculant}-(a), namely, it follows from Theorem \ref{thm_matrix_invertible_2}-(a) that
	\begin{equation*}
	S_n^{n-1} \subseteq \{(k,a_{k,1}^{(n)},b_{k,1}^{(n)})~|~k \in \R \setminus A_n \}.
	\end{equation*}
\end{remark}
Analogously, we also have
\begin{theorem}\label{thm_matrix_tn_circulant_odd}
	For $k \ne 0$ and an odd integer $n \geq 3,$ consider the associated circulant matrix
	\begin{equation*}
	B^{(k,a,b)}_n :=\left(\begin{array}{cccc} S^{(a,b)}_{k,1} & S^{(a,b)}_{k,2} & \cdots & S^{(a,b)}_{k,n} \\ S^{(a,b)}_{k,n} & S^{(a,b)}_{k,1} &\cdots & S^{(a,b)}_{k,n-1} \\ \vdots & \vdots & \vdots & \vdots \\ S^{(a,b)}_{k,2} & S^{(a,b)}_{k,3} & \cdots & S^{(a,b)}_{k,1}
	\end{array}\right).
	\end{equation*}
	For $1 \leq r \leq n-1,$ let $Y_r =\{M \in M_{n \times n}(\R)~|~\textrm{rank}(M) \leq r \}$ and let
	\begin{align*}
	T_n^r = & \{(k,a,b)\in (\R \setminus C_n) \times \R \times \R~|~B_n^{(k,a,b)} \in Y_r, E_{(a,b)}~\textrm{is singular}\} \\
	& \setminus \{(k,0,0)~|~k \in \R \}.
	\end{align*}
	Then we have:
	\begin{enumerate}
		\item[(a)] $T_n^{n-1}=\{(k,a_{k,3}^{(n)},b_{k,3}^{(n)})|~k \in \R \setminus C_n \}$.
		\item[(b)] A matrix $B_n^{(k,a,b)}$ has rank exactly $n-1$ for all but finitely many triples $(k,a,b)\in T_n^{n-1}$.
		\item[(c)] For each $1 \leq r \leq n-2,$ a matrix $B_n^{(k,a,b)}$ has rank $r$ for only finitely many triples $(k,a,b)\in T_n^{n-1}$.
	\end{enumerate}
\end{theorem}
\begin{proof}
	Imitate the proof of Theorem \ref{thm_matrix_sn_skewcirculant}, using Theorem \ref{thm_matrix_invertible_2}-(b).
\end{proof}
Finally, if $n \geq 4$ is even, then we have the following
\begin{theorem}\label{thm_matrix_tn_circulant_even}
	For $k \ne 0$ and an even integer $n \geq 4,$ let $B^{(k,a,b)}_n$ and $Y_r$ $(r=1,2,\cdots,n-1)$ be as in Theorem \ref{thm_matrix_tn_circulant_odd}.  For $1 \leq r \leq n-1$, let
	\begin{align*}
	U_n^r & = \{(k,a,b)\in (\R \setminus (B_n \cup C_n))\times \R \times \R~|~B_n^{(k,a,b)} \in Y_r, E_{(a,b)}~\textrm{is singular}\} \\
	& \quad \setminus \{(k,0,0)~|~k \in \R \}.
	\end{align*}
	Then we have:
	\begin{enumerate}
		\item[(a)] $U_n^{n-1}=\{(k,a_{k,2}^{(n)},b_{k,2}^{(n)})|~k \in \R \setminus (B_n \cup C_n) \} \cup \{(k,a_{k,3}^{(n)},b_{k,3}^{(n)})|~k \in \R \setminus (B_n \cup C_n) \}$.
		\item[(b)] A matrix $B_n^{(k,a,b)}$ has rank exactly $n-1$ for all but finitely many triples $(k,a,b)\in U_n^{n-1}$.
		\item[(c)] For each $1 \leq r \leq n-2,$ a matrix $B_n^{(k,a,b)}$ has rank $r$ for only finitely many triples $(k,a,b)\in U_n^{n-1}$.
	\end{enumerate}
\end{theorem}
\begin{proof}
	\begin{enumerate}
		\item[(a)] Note that $B_{n}^{(k,a,b)} \in Y_{n-1}$ if and only if $B_{n}^{(k,a,b)}$ is not invertible. Hence the result follows from Theorem \ref{thm_matrix_invertible_1}-(b), and the definition of $a_{k,2}^{(n)}, a_{k,3}^{(n)}, b_{k,2}^{(n)}$, and $b_{k,3}^{(n)}$.
		\item[(b)] Note that $B_{n}^{(k,a,b)} \in Y_{n-2}$ if and only if all the $(n-1) \times (n-1)$ minors of $B_{n}^{(k,a,b)}$ are 0. Also, we have $U_{n}^{n-2} \subseteq U_{n}^{n-1}.$ Thus, any element of $U_{n}^{n-2}$ is either of the form
		$(k,a^{(n)}_{k,2},b^{(n)}_{k,2})$ for some $k \in \R \setminus (B_{n} \cup C_n)$ where $a^{(n)}_{k,2}=-\frac{27(F_{n}(k)-1)^{2}}{4(G_{n}(k)-k+2)^{2}}~\textrm{and}~ b^{(n)}_{k,2}=\frac{27(F_{n}(k)-1)^{3}}{4(G_{n}(k)-k+2)^{3}}$ for some $F_{n},G_{n} \in \Z[T]$ or $(k,a^{(n)}_{k,3},b^{(n)}_{k,3})$ for some $k \in \R \setminus (B_n \cup C_n)$ where $a^{(n)}_{k,3}=-\frac{27(P_{n}(k)-1)^{2}}{4(Q_{n}(k)-k-2)^{2}}$ and $b^{(n)}_{k,3}=\frac{27(P_{n}(k)-1)^{3}}{4(Q_{n}(k)-k-2)^{3}}$ for some $P_{n},Q_{n} \in \Z[T]$. \\
		Let $(k,a^{(n)}_{k,2},b^{(n)}_{k,2}) \in U_{n}^{n-2}$ so that $B_{n}^{(k,a^{(n)}_{k,2},b^{(n)}_{k,2})} \in Y_{n-2}.$ By the above observation, it follows that we have a system of finitely many equations in $S_{k,1}^{(a^{(n)}_{k,2},b^{(n)}_{k,2})},\cdots,S_{k,n}^{(a^{(n)}_{k,2},b^{(n)}_{k,2})}$. Now, rewrite them in terms of $a_{k,2}^{(n)}$, $b_{k,2}^{(n)}$ according to definition, and by using the fact that $a^{(n)}_{k,2}=-\frac{27(F_{n}(k)-1)^{2}}{4(G_{n}(k)-k+2)^{2}}$ and $b^{(n)}_{k,2}=\frac{27(F_{n}(k)-1)^{3}}{4(G_{n}(k)-k+2)^{3}}$, we can see that there are only finitely many real solutions in $k$ for the given system of equations. \\
		Similarly, let $(k,a^{(n)}_{k,3}, b^{(n)}_{k,3}) \in U_{n}^{n-2}$ so that $B_{n}^{(k,a^{(n)}_{k,3},b^{(n)}_{k,3})} \in Y_{n-2}.$ By the above observation, it follows that we have a system of finitely many equations in $S_{k,1}^{(a^{(n)}_{k,3},b^{(n)}_{k,3})},\cdots,S_{k,n}^{(a^{(n)}_{k,3},b^{(n)}_{k,3})}$. Now, rewrite them in terms of $a_{k,3}^{(n)}$, $b_{k,3}^{(n)}$ according to definition, and by using the fact that $a^{(n)}_{k,3} = -\frac{27(P_{n}(k)-1)^{2}}{4(Q_{n}(k)-k-2)^{2}}$ and $b^{(n)}_{k,3}=\frac{27(P_{n}(k)-1)^{3}}{4(Q_{n}(k)-k-2)^{3}}$, we can see that there are only finitely many real solutions in $k$ for the given system of equations. \\
		Consequently, $U_{n}^{n-2}$ is finite and $B_{n}^{(k,a,b)}$ has rank exactly $n-1$ for all but $\sharp U_{n}^{n-2}$-many triples $(k,a,b)\in U_{n}^{n-1}.$
		\item[(c)] In the proof of (b), we have seen that $U_{n}^{n-2}$ is finite. For each $1 \leq r \leq n-2,$ we have $U_{n}^{r} \subseteq U_{n}^{n-2}$ so that $U_{n}^{r}$ is finite, and hence $B_{n}^{(k,a,b)}$ has rank $r$ for only $(\sharp U_{n}^{r} - \sharp U_{n}^{r-1})$-many triples $(k,a,b) \in U_{n}^{n-1}$.
	\end{enumerate}
This completes the proof.
\end{proof}
\begin{remark}
	In the case of $n=2$, parts (a) and (b) of Theorem \ref{thm_matrix_tn_circulant_even} still hold.
\end{remark}

\section*{Appendix} \label{sec_appen}
In this section, we introduce various formulas that are related to a given $k$-FL sequence. Before giving those formulas, we start with the following
\begin{remark_A}\label{rm_binet}
	It is well-known that the Binet's formulas for $\{F_{k,n}\}_{n \in \Z_{\geq 0}}$ and $\{L_{k,n}\}_{n \in \Z_{\geq 0}}$ are given by
	\begin{equation*}
	F_{k,n}=\frac{\alpha^n -\beta^n}{\alpha-\beta}~~\textrm{and}~~L_{k,n}=\alpha^n + \beta^n~~\textrm{for each}~n \geq 0
	\end{equation*}
	where $\alpha > \beta$ are the roots of the quadratic equation $x^2 -kx-1=0$.
\end{remark_A}
The Binet's formula for a $k$-FL sequence $\{S_{k,n}^{(a,b)}\}_{n \in \Z_{\geq 0}}$ is given by
\begin{lemma_A}[Binet's Formula]\label{lem_binet}
	For each integer $n \geq 0$, we have
	\begin{equation*}
	S^{(a,b)}_{k,n}=\frac{(a+\sqrt{k^2 +4}b)\alpha^n -(a-\sqrt{k^2 +4}b)\beta^n}{\alpha-\beta}
	\end{equation*}
	where $\alpha> \beta$ are the zeros of the polynomial $x^2-kx-1$.
\end{lemma_A}
\begin{proof}
	Note first that we have
	\begin{align*}
	S_{k,n}^{(a,b)} & =a \cdot  F_{k,n}+b \cdot L_{k,n}=a \cdot \frac{\alpha^n -\beta^n}{\alpha-\beta} + b \cdot (\alpha^n+\beta^n)\\
	& =\left(\frac{a}{\alpha-\beta}+b\right)\alpha^n - \left(\frac{a}{\alpha-\beta}-b\right)\beta^n
	\end{align*}
	for each $n \geq 0$, where $\alpha>\beta$ are the zeros of the polynomial $x^2-kx-1$ according to Remark \ref{rm_binet}. Since we have
	\begin{equation*}
	(\alpha-\beta)^2=(\alpha+\beta)^2-4\alpha \beta = k^2 +4
	\end{equation*}
	so that
	\begin{equation*}
	\alpha-\beta = \sqrt{k^2+4},
	\end{equation*}
	the result follows.
\end{proof}
Using Binet's formula, we may introduce some identities that relate a $k$-FL sequence and a $k$-Fibonacci sequence.
\begin{theorem_A}[Catalan's Identity]\label{thm_catalan}
	Let $\{S_{k,n}^{(a,b)}\}_{n \in \Z_{\geq 0}}$ be a $k$-FL sequence. For any two integers $n \geq r \geq 0,$ we have
	\begin{equation*}
	S_{k,n-r}^{(a,b)} S_{k,n+r}^{(a,b)} - (S_{k,n}^{(a,b)})^2 = (-1)^{n-r+1} \{a^2-(k^2+4)b^2\} F_{k,r}^2.
	\end{equation*}
\end{theorem_A}
\begin{proof}
	By Lemma \ref{lem_binet}, we can write
	\begin{equation*}
	S_{k,n}^{(a,b)}=\frac{X \alpha^n - Y \beta^n}{\alpha - \beta}
	\end{equation*}
	where $X=a+\sqrt{k^2+4}b$ and $Y=a-\sqrt{k^2+4}b$. Then we have
	\begin{align*}
	& S_{k,n-r}^{(a,b)} S_{k,n+r}^{(a,b)} - (S_{k,n}^{(a,b)})^2 \\
	= & \frac{1}{(\alpha-\beta)^2}\{(X \alpha^{n-r}-Y \beta^{n-r})(X \alpha^{n+r}-Y \beta^{n+r})-(X \alpha^n - Y \beta^n )^2 \}\\
	= & \frac{1}{(\alpha-\beta)^2}(-XY \alpha^{n-r}\beta^{n+r} -XY \alpha^{n+r} \beta^{n-r} +2XY \alpha^n \beta^n )\\
	= & \frac{-XY(\alpha \beta)^{n-r}}{(\alpha-\beta)^2}(\alpha^{2r}-2\alpha^r \beta^r +\beta^{2r})\\
	= & (-1)^{n-r+1} XY \left(\frac{\alpha^r - \beta^r}{\alpha-\beta}\right)^2.
	\end{align*}
	Now, since we have $XY=a^2 - (k^2+4)b^2$ and $F_{k,r}=\frac{\alpha^r - \beta^r}{\alpha-\beta}$, the result follows.
\end{proof}
In the sequel, we let $X=a+\sqrt{k^2+4}b$ and $Y=a-\sqrt{k^2+4}b$, as in the proof of Theorem \ref{thm_catalan}. The Cassini's identity is a special case of the Catalan's identity with $r=1$.
\begin{corollary_A}[Cassini's Identity]\label{cor_cassini}
	For any integer $n \geq 1$, we have
	\begin{equation*}
	S^{(a,b)}_{k,n-1}S^{(a,b)}_{k,n+1}-(S^{(a,b)}_{k,n})^2 = (-1)^n \{a^2 - (k^2+4)b^2 \}.
	\end{equation*}
\end{corollary_A}
\begin{proof}
	Recall that $F_{k,1}=1$.
\end{proof}
Corollary \ref{cor_cassini} has the following consequence:
\begin{corollary_A}
	A $k$-FL sequence $\{S_{k,n}^{(a,b)}\}_{n \in \Z_{\geq 0}}$ is a geometric sequence if and only if $a^2 = (k^2 +4)b^2$. In particular, given a positive real number $k$, there are infinitely many pairs $(a,b)\in \R^2$ such that $\{S_{k,n}^{(a,b)}\}_{n \in \Z_{\geq 0}}$ is a geometric sequence.
\end{corollary_A}
Now, we are ready to move to
\begin{theorem_A}[d$^\prime$Ocagne's Identity]\label{thm_docagne}
	Let $\{S_{k,n}^{(a,b)}\}_{n \in \Z_{\geq 0}}$ be a $k$-FL sequence. For any two integers $m \geq n \geq 0$, we have
	\begin{equation*}
	S_{k,m}^{(a,b)} S_{k,n+1}^{(a,b)} - S_{k,m+1}^{(a,b)} S_{k,n}^{(a,b)} = (-1)^{n} \{a^2-(k^2+4)b^2\} F_{k,m-n}.
	\end{equation*}
\end{theorem_A}
\begin{proof}
	By Lemma \ref{lem_binet}, we have
	\begin{align*}
	& S_{k,m}^{(a,b)} S_{k,n+1}^{(a,b)} - S_{k,m+1}^{(a,b)} S_{k,n}^{(a,b)}\\
	= &\frac{1}{(\alpha-\beta)^2}\{(X \alpha^{m}-Y \beta^{m})(X \alpha^{n+1}-Y \beta^{n+1}) \\
	& \quad -(X \alpha^{m+1}-Y \beta^{m+1})(X \alpha^{n}-Y \beta^{n})\}\\
	= & \frac{-XY}{(\alpha-\beta)^2}(\alpha^m \beta^{n+1}+\alpha^{n+1}\beta^m - \alpha^{m+1}\beta^n - \alpha^n \beta^{m+1})\\
	= & \frac{-XY}{(\alpha-\beta)^2}\{(\alpha\beta)^n \beta^{m-n}(\alpha-\beta)-(\alpha\beta)^n \alpha^{m-n}(\alpha-\beta)\}\\
	= & \frac{-(\alpha\beta)^n XY (\alpha-\beta)}{(\alpha-\beta)^2} (\beta^{m-n}-\alpha^{m-n})\\
	= & (-1)^n XY \frac{\alpha^{m-n}-\beta^{m-n}}{\alpha-\beta}\\
	= & (-1)^n \{a^2-(k^2+4)b^2\}F_{k,m-n}.
	\end{align*}
	This completes the proof.
\end{proof}
Next, we are interested in finding various types of sums of terms of a $k$-FL sequence. Again, let $\{S_{k,n}^{(a,b)}\}_{n \in \Z_{\geq 0}}$ be a $k$-FL sequence.
\begin{theorem_A}\label{thm_identity_sum_all}
	For any integer $n \geq 1$, we have
	\begin{equation*}
	\sum_{i=1}^n S_{k,i}^{(a,b)}=\frac{1}{k}\left(S_{k,n+1}^{(a,b)}+S_{k,n}^{(a,b)}\right)-b -\frac{a+2b}{k}.
	\end{equation*}
\end{theorem_A}
\begin{proof}
	By Lemma \ref{lem_binet}, we have
	\begin{align*}
	& \sum_{i=1}^n S_{k,i}^{(a,b)} \\
	= & \frac{1}{\alpha-\beta} \sum_{i=1}^n (X \alpha^i - Y \beta^i ) = \frac{1}{\alpha-\beta}\left(\frac{X \alpha(\alpha^n -1)}{\alpha-1}- \frac{Y \beta(\beta^n -1)}{\beta-1} \right)\\
	= & \frac{1}{(\alpha-\beta)(\alpha-1)(\beta-1)}\left(X \alpha(\alpha^n -1)(\beta-1) - Y \beta(\beta^n -1)(\alpha-1)\right)\\
	= & \frac{-1}{(\alpha-\beta)(\alpha+\beta)}\left(-X (\alpha^n -1)(\alpha+1)+Y(\beta^n -1)(\beta+1)\right)\\
	= & \frac{1}{k} \left(\frac{X \alpha^{n+1}-Y \beta^{n+1}}{\alpha-\beta}+\frac{X \alpha^{n}-Y \beta^{n}}{\alpha-\beta}- \frac{X \alpha-Y \beta}{\alpha-\beta}- \frac{X -Y }{\alpha-\beta} \right)\\
	= & \frac{1}{k} \left(S_{k,n+1}^{(a,b)}+S_{k,n}^{(a,b)} - S_{k,1}^{(a,b)}-S_{k,0}^{(a,b)}\right)= \frac{1}{k}\left(S_{k,n+1}^{(a,b)}+S_{k,n}^{(a,b)}\right)-b-\frac{a+2b}{k}.
	\end{align*}
	This completes the proof.
\end{proof}
Another possible result is the following
\begin{theorem_A}\label{thm_identity_sum_even}
	For any integer $n \geq 1$, we have
	\begin{equation*}
	\sum_{i=1}^n S_{k,2i}^{(a,b)} = \frac{1}{k} \cdot S_{k,2n+1}^{(a,b)}-b-\frac{a}{k}.
	\end{equation*}
\end{theorem_A}
\begin{proof}
	By Lemma \ref{lem_binet}, we have
	\begin{align*}
	& \sum_{i=1}^n S_{k,2i}^{(a,b)} \\
	= & \frac{1}{\alpha-\beta} \sum_{i=1}^n (X \alpha^{2i} - Y \beta^{2i} ) = \frac{1}{\alpha-\beta}\left(\frac{X \alpha^2 (\alpha^{2n} -1)}{\alpha^2 -1}- \frac{Y \beta^2 (\beta^{2n} -1)}{\beta^2 -1} \right)\\
	= & \frac{1}{(\alpha-\beta)(\alpha^2 -1)(\beta^2 -1)}\left(X \alpha^2 (\alpha^{2n} -1)(\beta^2 -1) - Y \beta^2 (\beta^{2n} -1)(\alpha^2 -1)\right)\\
	= & \frac{-1}{(\alpha-\beta)k^2 }\left(X (\alpha^{2n} -1)(1-\alpha^2)- Y(\beta^{2n} -1)(1-\beta^2 )\right)\\
	= & \frac{1}{(\alpha-\beta)k }\left(X (\alpha^{2n+1} -\alpha)- Y(\beta^{2n+1} -\beta) \right)\\
	= & \frac{1}{k} \left(\frac{X \alpha^{2n+1}-Y \beta^{2n+1}}{\alpha-\beta} - \frac{X \alpha-Y \beta}{\alpha-\beta} \right)\\
	= & \frac{1}{k} \left(S_{k,2n+1}^{(a,b)} - S_{k,1}^{(a,b)}\right)= \frac{1}{k} \cdot S_{k,2n+1}^{(a,b)}-b-\frac{a}{k}.
	\end{align*}
	This completes the proof.
\end{proof}
Two immediate consequences of Theorem \ref{thm_identity_sum_all} and Theorem \ref{thm_identity_sum_even} are the following corollaries:
\begin{corollary_A}\label{cor_identity_sum_odd}
	For any integer $n \geq 1$, we have
	\begin{equation*}
	\sum_{i=1}^n S_{k,2i-1}^{(a,b)}=\frac{1}{k} \cdot S_{k,2n}^{(a,b)}-\frac{2b}{k}.
	\end{equation*}
\end{corollary_A}
\begin{proof}
	By Theorem \ref{thm_identity_sum_all} and Theorem \ref{thm_identity_sum_even}, we have
	\begin{align*}
	\sum_{i=1}^n S_{k,2i-1}^{(a,b)} & = \sum_{i=1}^{2n}S_{k,i}^{(a,b)}-\sum_{i=1}^n S_{k,2i}^{(a,b)} \\
	& =\frac{1}{k}(S_{k,2n+1}^{(a,b)}+S_{k,2n}^{(a,b)})-b-\frac{a+2b}{k}-\frac{1}{k} \cdot S_{k,2n+1}^{(a,b)}+b+\frac{a}{k}\\
	& =\frac{1}{k} \cdot S_{k,2n}^{(a,b)}-\frac{2b}{k}.
	\end{align*}
	This completes the proof.
\end{proof}
\begin{corollary_A}\label{cor_-1_sum}
	For any integer $n \geq 1$, let $r=\left[\frac{n}{2}\right]$. Then we have
	\begin{equation*}
	\sum_{i=1}^n (-1)^i \cdot S_{k,i}^{(a,b)}=\frac{1}{k}\left(2 S_{k,2r+1}^{(a,b)}-S_{k,n+1}^{(a,b)}-S_{k,n}^{(a,b)}\right) -b -\frac{a-2b}{k}.
	\end{equation*}
\end{corollary_A}
\begin{proof}
	Note that we have
	\begin{equation*}
	\sum_{i=1}^n (-1)^i \cdot S_{k,i}^{(a,b)}=2 \sum_{i=1}^r S_{k,2i}^{(a,b)}-\sum_{i=1}^n S_{k,i}^{(a,b)}.
	\end{equation*}
	Now, by using Theorem \ref{thm_identity_sum_all} and Theorem \ref{thm_identity_sum_even}, the result follows.
\end{proof}
There are also some results about the sums involving binomial coefficients in the literature. Regarding this kind of sums, we have the following
\begin{theorem_A}\label{thm_identity_binomial}
	For any integer $n \geq 0$, we have
	\begin{equation*}
	\sum_{i=0}^n \binom{n}{i}  k^i \cdot S_{k,i}^{(a,b)}=S_{k,2n}^{(a,b)}.
	\end{equation*}
\end{theorem_A}
\begin{proof}
	By Lemma \ref{lem_binet}, we have
	\begin{align*}
	\sum_{i=0}^n \binom{n}{i}  k^i \cdot S_{k,i}^{(a,b)} & = \frac{1}{\alpha-\beta}\sum_{i=0}^n \binom{n}{i}  k^i (X \alpha^i -Y \beta^i ) \\
	& =\frac{1}{\alpha-\beta} \left(X \sum_{i=0}^n \binom{n}{i}  (k \alpha)^i - Y \sum_{i=0}^n \binom{n}{i}  (k \beta)^i \right) \\
	& =\frac{1}{\alpha-\beta} \{X (k \alpha+1)^n - Y (k \beta+1)^n \} \\
	& =\frac{1}{\alpha-\beta}(X \alpha^{2n}-Y \beta^{2n})~~~(\textrm{since }k\alpha+1=\alpha^2, k\beta+1=\beta^2) \\
	& =S_{k,2n}^{(a,b)}.
	\end{align*}
	This completes the proof.
\end{proof}
Finally, we examine another interesting identity:
\begin{theorem_A}[Livio's Formula]\label{thm_identity_Livio}
	For any real number $p$ with $p \ne 0,\alpha, \beta$, and for any integer $n \geq 1$, we have
	\begin{equation*}
	\sum_{i=1}^n \frac{S_{k,i}^{(a,b)}}{p^i}=\frac{1}{(p^2 -pk -1)p^n }\left(-S_{k,n}^{(a,b)}-p S_{k,n+1}^{(a,b)}+p^{n+1}(a+bk)+2b p^n \right).
	\end{equation*}
\end{theorem_A}
\begin{proof}
	By Lemma \ref{lem_binet}, we have
	\begin{align*}
	& \sum_{i=1}^n \frac{S_{k,i}^{(a,b)}}{p^i} = \frac{1}{\alpha-\beta} \left(X \cdot \sum_{i=1}^n \left(\frac{\alpha}{p}\right)^i - Y \cdot \sum_{i=1}^n \left(\frac{\beta}{p}\right)^i \right)\\
	= & \frac{1}{\alpha-\beta}\left(X \cdot \frac{\alpha \left(\frac{\alpha}{p}\right)^n -\alpha}{\alpha-p}- Y \cdot \frac{\beta \left(\frac{\beta}{p}\right)^n -\beta}{\beta-p} \right)\\
	= & \frac{1}{(\alpha-\beta)(\alpha-p)(\beta-p)}\left\{X \left( \alpha \left(\frac{\alpha}{p}\right)^n -\alpha \right)(\beta-p) \right.\\
	& \quad \left. - Y \left( \beta \left(\frac{\beta}{p}\right)^n -\beta \right)(\alpha-p)\right\}\\
	= & \frac{1}{(\alpha-\beta)(p^2-pk-1)p^n}\left\{X (\alpha^n -p^n)(-1-p \alpha)-Y(\beta^n -p^n)(-1-p\beta)\right\}\\
	= & \frac{1}{(p^2-pk-1)p^n} \left(-\frac{X \alpha^{n}-Y \beta^{n}}{\alpha-\beta}-p \cdot \frac{X \alpha^{n+1}-Y \beta^{n+1}}{\alpha-\beta} \right.\\
	& \quad \left. + p^{n+1} \cdot \frac{X \alpha-Y \beta}{\alpha-\beta} + p^n \cdot \frac{X -Y }{\alpha-\beta} \right)\\
	= & \frac{1}{(p^2-pk-1)p^n} \left(-S_{k,n}^{(a,b)}-p S_{k,n+1}^{(a,b)} +p^{n+1} S_{k,1}^{(a,b)} +p^n S_{k,0}^{(a,b)}\right)\\
	= &\frac{1}{(p^2-pk-1)p^n}\left(-S_{k,n}^{(a,b)}-p S_{k,n+1}^{(a,b)}+p^{n+1}(a+bk)+2b p^n \right).
	\end{align*}
	This completes the proof.
\end{proof}
\begin{remark_A}
	Theorem \ref{thm_identity_Livio} gives another proof of Corollary \ref{cor_-1_sum} by taking $p=-1$.
\end{remark_A}

\medskip

\noindent MSC2010: 11B39, 15B99, 14M12, 14H50

\end{document}